\documentclass[10pt,a4paper]{amsart}
\usepackage{graphicx}
\usepackage{alltt}
\usepackage{amsmath,amstext, amsfonts,amssymb,amsthm}
\usepackage[hidelinks, pdftex]{hyperref}

\newcommand{\doi}[1]{\href{https://doi.org/#1}{\texttt{https://doi.org/#1}}}

\newcommand{\RR}{\mathbb R}
\newcommand{\NN}{\mathbb N}
\newcommand{\ZZ}{\mathbb Z}
\newcommand{\EE}{\mathcal{E}}

\usepackage{appendix}
\usepackage{mathrsfs, enumerate}
\newtheorem{theorem}{Theorem}
\newtheorem{lemma}{Lemma}
\newtheorem{remark}{Remark}
\newtheorem{definition}{Definition}

\allowdisplaybreaks

\begin{document}

\begin{center}
\LARGE
\textbf{Fixed Point Theorems in M-distance Spaces}

\small 
\textbf {Vladyslav Babenko, Vira Babenko, Oleg Kovalenko}\\[6pt]
Department of Mathematical Analysis and Theory of Functions, Oles Honchar Dnipro National University, Dnipro, Ukraine, babenko.vladislav@gmail.com \\

Department of Mathematics and Computer Science, Drake University, Des Moines, USA, vira.babenko@drake.edu \\

Department of Mathematical Analysis and Theory of Functions, Oles Honchar Dnipro National University, Dnipro, Ukraine,  olegkovalenko90@gmail.com
\end{center}

\begin{abstract}
We prove fixed point theorems in a space with a distance function that takes values in a partially ordered monoid. On the one hand, such an approach allows one to generalize some fixed point theorems in a broad class of spaces, including metric and uniform spaces. On the other hand, compared to the so-called cone metric spaces and $K$-metric spaces, we do not require that the distance function range has a linear structure. We also consider several applications of the obtained fixed point theorems. In particular, we consider the questions of the existence of solutions of the Fredholm integral equation in $L$-spaces. \vskip 2mm

\end{abstract}

\section{Introduction}
The theory of fixed point theorems is a well developed domain of Analysis and Topology (see books~\cite{Granas,Agarwal,Kirk} and references therein). It has a numerous amount of applications, in particular in  Numerical Analysis, the Theory of ODEs and PDEs, Integral Equations, Mathematical Economics and others. 

One of the most well known fixed point theorems is the contraction mapping theorem in metric spaces, which goes back to Picard, Banach and Caccioppoli. The contraction mapping principle was generalized in various directions, see, e.g. mentioned above books and surveys~\cite{Zabrejko, Jancovic}.
 A majority of the contraction mapping theorems use the following scheme of the proof. Consider an appropriate sequence 
\begin{equation}\label{sequence}
x_0,x_1=f(x_0),\ldots,x_n=f(x_{n-1}),\ldots
\end{equation}
and establish that it is a Cauchy sequence. Using completeness of the space, obtain a point $x$, such that $x_n\to x$, as $n\to\infty.$ Continuity (in some sense) of $f$ gives $f(x)=x$. 

This scheme can be implemented not only in a metric space, but also in a space with agreeing notions of convergent and Cauchy sequences, and that satisfies some properties that allow us to prove that \eqref{sequence} is a Cauchy sequence. One  example of such spaces are uniform spaces, see e.g.~\cite[Ch. 8]{Engelking}~\cite[Ch. 6]{Kelley}. It seems that the first contraction mapping type result in uniform spaces was obtained in~\cite{Taylor}. Another important class of spaces that satisfies the stated above requirements is the class of so-called distance spaces, or symmetries.
Some results regarding the fixed point theorems in real-valued distance spaces can be found in~\cite{Kirk, Choban, ChobanBerinde} (see also references therein).

We consider distance spaces with distances that take values in a partially ordered monoid, so that there might be no linear structure in the range of the distance function. 

Note, that spaces with metric that takes values in partially ordered vector spaces were introduced by Kurepa~\cite{Kurepa} in~1934. Later Kantorovich developed a theory of normed spaces with norms that take values in complete vector latices, see~\cite{Kantorovich}. These objects appeared to be fruitful in the study of functional equations using iterative methods, and in related questions of Analysis, see~\cite{KVP}. 

 Results on the fixed point theorems for the metric spaces with metric that takes values in a partially ordered vector spaces can be found in the surveys~\cite{Zabrejko, Jancovic}. In particular, in these articles the question whether results on cone metric or $K$-metric spaces can be reduced to results on ordinary metric spaces, is discussed. In~\cite{Radenovic_Kadelburg} fixed point theorems in  cone distance spaces were considered.

{Note, that each uniform space with separating axiom can be metrized with a metric that takes values in a partially ordered monoid, see  Section~\ref{s::applications}. As it is well known, metrization of a uniform space with a real-valued metric is possible only for the uniform spaces with countable base. Questions of metrization of a uniform spaces with $K$-space valued metrics were considered in~\cite{Kusraev}.}

In the case of functions acting in partially ordered spaces, the continuity property of the functions can be substituted by their monotonicity. One of the first fixed point theorems of this kind was proved by Tarski~\cite{tarski}.  For more results in this area we refer the reader to~\cite{guo2004partial,carl2014fixed, Jancovic,Reem} and references therein. A related area of research is the study of multiple fixed points, see~\cite{ ChobanBerinde} and references therein.

The article is organized as follows. 
In Section~\ref{s::MDistanceSpaces} we adduce necessary definitions of a partially ordered monoid $M$, $M$-valued distance functions, the notions of a series and their convergence in $M$. In order to introduce convergence in an $M$-distance space, we define a family $N(M)$ of null sequences in the monoid $M$. We also consider an example of a special family of null sequences $N_{\EE}(M)$ and examples of $M$-distance spaces.

In Section~\ref{s::contractionTheorem} we prove fixed point theorems for the mappings defined on $M$-distance spaces and partially ordered $M$-distance spaces. We also prove a multiple fixed point theorem.

In Section~\ref{s::applications} we consider several applications of the obtained fixed point theorems. In particular, we consider the 
 questions of existence of
solutions of the Fredholm integral equation in $L$-spaces {(i.~e. semi-linear metric spaces with two additional axioms, which connect the metric with the algebraic operations), see~\cite{Vahrameev, Babenko19, Babenko20}. We obtain significantly more general than in~\cite{Babenko19} conditions for existence and uniqueness of a solution of the  Fredholm integral equation.
}

Finally, the Appendix contains proofs of two lemmas, which are not needed for the main results of the article but seem to be illustrative.

\section{M-distance spaces}\label{s::MDistanceSpaces}
\subsection{Partially ordered monoids. Null sequences. Convergent and Cauchy series.}
\begin{definition} (see~\cite[Chapter 1,\S 1]{Lang})
A set $M$ with associative binary operation $+$ is called a monoid, if there exists  $\theta_M\in M$ such that for all $x\in M$
$$
\theta_M+x=x=x+\theta_M.
$$
\end{definition}

\begin{definition}
A set $X$ is called partially ordered if for some pairs of elements $x,y\in X$ the relation $x\le y$  is defined and the following properties are satisfied:
\begin{enumerate}
    \item For all $x\in X$ one has $x\le x$;
    \item If $x\le y$ and $y\le z$, then $x\le z$;
    \item If $x\le y$ and $y\le x$, then $x= y$.
\end{enumerate}
\end{definition}

\begin{definition}
{A monoid $M$ is called a partially ordered monoid, if it is a partially ordered set,} $$M_+:=\{ x\in M\colon \theta_M\le x\}\neq \{\theta_M\},$$ and the following condition holds: 
$$
    \text{If } x_1\leq y_1\, \text{and}\, x_2\leq y_2 \text{ then } x_1+x_2\leq y_1+y_2.
$$
\end{definition}
\begin{definition}
We say that the upper Riesz property holds in a partially ordered space $X$, if for each pair of elements $x,y\in X$  there exists an element $x\vee y\in X$ (which is called the supremum of $x$ and $y$) that satisfies the following properties:
\begin{enumerate}
    \item $x\leq x\vee y$ and $y\leq x\vee y$;
    \item If $z\in X$ is such that $x\leq z$ and $y\leq z$, then $x\vee y\leq z$.
\end{enumerate}
Analogously one can define the infimum $x\wedge y$ and the lower Riesz property.
\end{definition}

\begin{definition}\label{def::nullSequences}
Let $M$ be a partially ordered monoid. Denote by $N(M)$ a family of sequences 
$
\{x_1,\ldots, x_n,\ldots\}\subset M_+
$
such that the following properties hold:
\begin{enumerate}
 \item $\{\theta_M,\ldots, \theta_M,\ldots\}\in N(M)$ and if $x\in M_+$, $\{x,\ldots, x,\ldots\}\in N(M)$, then $x=\theta_M$.
    \item If $\{x_n\}, \{y_n\}\in N(M)$, then $\{x_n+y_n\}\in N(M)$.
    \item If $\{x_n\}\in N(M)$ and $\theta_M\leq y_n\leq x_n$ for all $n\in \NN$, then $\{y_n\}\in N(M)$.
    \item Substitution, addition or removal of a finite number of elements of a sequence from $N(M)$ preserves inclusion into 
    $N(M)$.
     \item Each subsequence $\{x_{n_k}\}$ of a sequence $\{x_n\}\in N(M)$ belongs to $N(M)$.
\end{enumerate}
$N(M)$ will be called a family of null sequences of $M$.
\end{definition}

We also need the notion of a series in a partially ordered monoid.
\begin{definition}
For a sequence $\{x_n\}\subset M_+$ the symbol 
\begin{equation}\label{seriesDef}
 \sum\limits_{k=1}^\infty x_k
\end{equation}
is called a series. The sum $s_n=\sum_{k=1}^nx_k$, $n\in \NN$, is called a partial sum of the series.

We say that series~\eqref{seriesDef} converges, if there exist an element $s\in M_+$ and a null sequence $\{r_n\}$, such that for all $n\in \mathbb{N}$ 
\[
s=s_n+r_n.
\]

We say that series~\eqref{seriesDef} is a Cauchy series, 
if for all increasing sequences of natural numbers $\{n_k\}$ and $\{m_k\}$ such that $m_k\geq n_k$ for all $k\in\NN$, one has
$$
\left\{\sum_{s=n_k}^{m_k} x_s\right\} \in N(M).
$$
\end{definition}
{
\begin{remark}
The notion of a Cauchy series is enough for our purposes. 
However, since in many situations the question of convergence of a series is well studied, we defined the notion of convergent series in a partially ordered monoid as well.
\end{remark}
}

Assume that a partially ordered monoid $M$ {is commutative and}  satisfies the following additional property: if $x,y,z\in M$ are such that $x\neq y$, then  $x+z\neq y+z$.

The difference of the elements  $x$ and $y$  of the monoid $M$ is called the element $z\in M$ such that $x=y+z$, if it exists. The difference of the elements $x$ and $y$ is denoted by $x\ominus y$. The above additional property  implies that the difference of two elements $x$ and $y$ is unique, provided it exists.

From the definition the following property   follows immediately: if the differences  $x\ominus y$ and $y\ominus z$ exist, then the difference $x\ominus z$ exists and 
\begin{equation}\label{hukuharuDiffsSum}
    x\ominus z=(x\ominus y)+(y\ominus z).
\end{equation}

Using the notion of difference, the definition of a convergent series can be restated as follows. Series~\eqref{seriesDef}  converges, if there exists an element $s\in M$, such that for all $n\in \mathbb{N}$ the difference $s\ominus s_n\in M_+$ exists and $\{s\ominus s_n\}\in N(M)$.
Note, that for series~\eqref{seriesDef}, the difference $s_m\ominus s_n$ of its partial sums exists for all $m\geq n$ and is equal to $\sum_{k=n+1}^m x_k$. 

\begin{lemma}
If a sequence $\{x_n\}\subset M_+$ is such that the series
$\sum_{n=1}^\infty x_{n}$ converges, then this series is a Cauchy series.
\end{lemma}

\begin{proof}

For a convergent series, due to~\eqref{hukuharuDiffsSum}, for all $m \geq n$ one has
$$
s_m\ominus
s_n=(s\ominus s_n)\ominus(s\ominus s_m),
$$
where $s$ is the sum of the series.
For arbitrary increasing sequences of natural numbers $\{n_k\}$ and $\{m_k\}$ such that $m_k\geq n_k$ for all $k\in\NN$,
$$
s_{m_k}\ominus s_{n_k} 
\leq
(s\ominus s_{m_k}) + (s_{m_k}\ominus s_{n_k})
=
s\ominus s_{n_k}\in N(M).
$$
The lemma is proved. 
\end{proof}

\subsection{M-valued distance}
Let $X$ be a set and $M$ be a partially ordered monoid. 
{
\begin{definition}
A mapping $d_{X,M}(\cdot,\cdot)\colon X\times X\to M_+$ is called an $M$-valued distance in $X$, if $d_{X,M}(x,y) = d_{X,M}(y,x)$ for all $x,y\in X$ and 
$$
 d_{X,M}(x,y)=\theta_M\text{ if and only if } x=y.
$$
The pair $(X,d_{X,M})$ is called an $M$-distance space.
\end{definition}
}

\begin{definition}
We say that a sequence $\{ x_n\}\subset X$ converges to $x\in X$, and write $x_n\to x$ as $n\to \infty$, if  $\{d_{X,M}(x_n,x)\}\in N(M)$.
\end{definition}
\begin{definition}
A sequence $\{ x_n\}\subset X$ is called a Cauchy sequence, if for each increasing sequences of natural numbers $\{n_k\}$ and $\{m_k\}$ such that $m_k> n_k$ for all $k\in\NN$,  one has $\left\{d_{X,M}(x_{n_k},x_{m_k})\right\}\in N(M).$
\end{definition}

We need the following generalization of the Fr\'{e}chet-Wilson property (see~\cite[pp. 363-364]{Wilson})
\begin{definition}
An $M$-distance space $(X,d_{X,M})$ is said to satisfy the  Fr\'{e}chet-Wilson property, if for any sequences $\{ x_n\},\{ y_n\},\{ z_n\}\subset X$, 
$$
\{d_{X,M}(x_n,z_{n})+d_{X,M}(z_{n},y_n)\}\in N(M)\implies\{d_{X,M}(x_n,y_{n})\}\in N(M).
$$
\end{definition}

\begin{lemma}
{If an $M$-distance space $(X,d_{X,M})$ satisfies the  Fr\'{e}chet-Wilson property, then } the following properties hold.
\begin{enumerate}
    \item Each converging sequence has a unique limit.
    \item Each subsequence of a converging sequence converges to the same limit.
    \item Each converging sequence is a Cauchy sequence.
\end{enumerate}
\end{lemma}
\begin{proof}
Assume that $x_n\to x$ and $x_n\to y$ as $n\to\infty$. Then $$
\{d_{X,M}(x,x_{n})+d_{X,M}(x_{n},y)\}\in N(M).
$$
Due to the Fr\'{e}chet-Wilson property,  $\{ d_{X,M}(x,y)\}\in N(M)$, hence  $d_{X,M}(x,y) = \theta_M$, and $x = y$.

The second property follows from corresponding property of the family $N(M)$ of null sequences.

Let $\{u_k\}$, $\{v_k\}$ be increasing sequences of natural numbers, $u_k<v_k$ for all $k\in \NN$, and $x_n\to x$ as $n\to\infty$. We have
$$
\{d_{X,M}(x,x_{u_k})+d_{X,M}(x_{v_k},y)\}\in N(M).
$$
Therefore $\{d_{X,M}(x_{u_k},x_{v_k})\}\in N(M)$.
\end{proof}

\begin{definition}
We say that an $M$-distance space $(X,d_{X,M})$ satisfies the strong Fr\'{e}chet-Wilson property, if for arbitrary sequence $\{x_n\}\subset X$,  and increasing sequences of natural numbers $\{n_k\}$ and $\{m_k\}$ such that $m_k> n_k$ for all $k\in\NN$, 
$$
\left\{\sum_{s = n_k}^{m_k-1} d_{X,M}(x_s,x_{s+1})\right\}\in N(M)\implies\left\{ d_{X,M}(x_{n_k},x_{m_k})\right\}\in N(M).
$$
$M$-distance spaces $(X,d_{X,M})$ that satisfy the strong Fr\'{e}chet-Wilson property will be called $FM$-spaces.
\end{definition}
{
\begin{remark}
{ 
It is easy to see, that the strong Fr\'{e}chet-Wilson property implies the Fr\'{e}chet-Wilson property. However, as Example~3 from Section~\ref{s::examples} shows, these properties are not equivalent.
}
\end{remark}
}

The following statement is obvious.

\begin{lemma}\label{lem::connection}
If in an $FM$-distance space a sequence $\{x_n\}\subset X$ is such that 
$$
    \sum_{k = 1}^{\infty} d_{X,M}(x_k,x_{k+1})
$$
is a Cauchy series, then $\{x_n\}$ is a Cauchy sequence.
\end{lemma}
\subsection{An example of a family of null sequences. E-convergence}
\begin{definition}\label{def::solidM+}
    Assume that the set  $M_+$ of a partially ordered monoid $M$ contains a non-empty set $\mathcal{E}\subset M_+^0:= M_+\setminus \{ \theta_M \}$ that satisfies the following conditions:
    \begin{enumerate}
        \item If $x\in M_+$ is such that for all $\varepsilon\in \mathcal{E}$ one has $x\le \varepsilon$, then $x=\theta_M$.
    \item For each $\varepsilon\in \mathcal{E}$ there exists $\delta\in \mathcal{E}$ such that $\delta + \delta \leq \varepsilon$.
    \end{enumerate}
Define a family of null sequences $N_{\mathcal{E}}(M)$ as the family of sequences  $\{x_n\}\subset M_+$ with the following property:  
for any $\varepsilon\in \mathcal{E}$ there  exists $N \in \NN$ such that  $x_n < \varepsilon$ for all $n\geq N$.

\end{definition}

The verification of all properties from Definition~\ref{def::nullSequences} can be done analogously to the real limit properties. 

If the family of null sequences is chosen to be $N_{\EE}(M)$, then some notions introduced in previous subsections can be written in a slightly simpler way. Lemma~\ref{l::econvergence} is obvious. Proofs of Lemmas~\ref{l::eCauchySeries} and~\ref{l::eFMCondition}  will be given in Appendix~\ref{a::eNullSeqFamily}.

\begin{lemma}\label{l::econvergence}
If $N(M) = N_{\EE}(M)$, then for a sequence $\{ x_n\}\subset X$ and $x\in X$ one has that $x_n\to x$ {as} $n\to\infty$ iff for any $\varepsilon\in \EE$ there exists $N\in\NN$ such that for $n\geq N$ one has $d_{X,M}(x_n,x)< \varepsilon$.
\end{lemma}
\begin{remark}
Such a convergence is called $\EE$-convergence (cf.~\cite{Zabrejko}).
\end{remark}

\begin{lemma}\label{l::eCauchySeries}
If $N(M) = N_{\EE}(M)$, then a series $\sum_{n=1}^\infty x_n$, $\{x_n\}\subset M_+$, is a Cauchy series if and only if for all $\varepsilon\in \EE$ there exists $N\in\NN$ such that for all $m\geq n\geq N$ one has
\begin{equation}\label{eFundamelity}
    \sum_{k=n}^m x_k < \varepsilon.
\end{equation}
\end{lemma}
 
 \begin{lemma}\label{l::eFMCondition}
 Let $N(M) = N_{\EE}(M)$. For an $M$-distance space $(X,d_{X,M})$ to be an $FM$-distance space it is sufficient, and 
 if there exists a null sequence $\{\varepsilon_n\}\subset\EE$,
 then it is necessary, that for any $\varepsilon\in \mathcal{E}$ there exists $\delta\in\mathcal{E}$ such that for all $n\in\NN$ and $x_1,\ldots, x_n\in X$ the inequality 
$$
\sum\limits_{k=1}^{n-1} d_{X,M}(x_k,x_{k+1})< \delta
$$
implies $d_{X,M}(x_1,x_n)< \varepsilon$.
 \end{lemma}
The proof of the following lemma is similar to the proof of Lemma~\ref{l::eCauchySeries}.
 \begin{lemma}
  Let $N(M) = N_{\EE}(M)$. A sequence $\{x_n\}$ in an $FM$-distance space is a Cauchy sequence if and only if for all $\varepsilon\in\EE$ there exists $N\in\NN$ such that $d_{X,M}(x_n,x_m)<\varepsilon$ for all $m\geq n\geq N$.
 \end{lemma}

\subsection{Examples of FM-distance spaces.}\label{s::examples}
Below we adduce several examples of $FM$-distance spaces. 

{\bf 1. $M$-metric spaces.}
    Let $X$ be a set, $M$ be a partially ordered monoid. 
A mapping $h_{X,M}(\cdot,\cdot)\colon X\times X\to M_+$ is called an $M$-valued metric (or $M$-metric), if the following conditions hold:
\begin{enumerate}
    \item For all $x,y\in X$, $x=y$ if and only if $ h_{X,M}(x,y)=\theta_M$;
     \item  For all $x,y\in X$, $h_{X,M}(x,y)=h_{X,M}(y,x)$;
    \item For all $x,y,z\in X$, $h_{X,M}(x,y)\le h_{X,M}(x,z)+h_{X,M}(z,y)$.
\end{enumerate}
The pair $(X,h_{X,M})$ is called an $M$-metric space. 
Each $M$-metric space is an $FM$-distance space.

   {\bf 2. $K$-metric, $K$-normed, and Cone metric spaces} are partial cases of $FM$-distance spaces (see for details~\cite{Zabrejko,Jancovic} and references therein). 
   
   {\bf 3. $FM$-distance spaces that are not metric spaces.}
The pair $(\mathbb{R}, d_{\mathbb{R},\mathbb{R}})$, where
    $$
        d_{\RR,\RR}(x,y) =
        \begin{cases}
         |x-y|,& |x-y|\leq 1,\\
         (x-y)^2, & |x-y|>1
        \end{cases}
    $$
    is an example of an $FM$-distance space that is not a metric space.
    
    The pair $(\mathbb{R}, (x,y)\to (x-y)^2)$, 
    is an example of an $M$-distance space that satisfies the  Fr\'{e}chet-Wilson property, but does not satisfy the strong Fr\'{e}chet-Wilson property.
    
     Let $C(T ,X)$ be the space of continuous functions defined on a metric compact $T$ with values in an $L$-space $X$.
    Definitions and necessary facts from the theory of $L$-spaces, in particular the definition of the Lebesgue integral for an $L$-spaces valued functions that will be needed in Section~\ref{s::fredholmEquation},  can be found in~\cite{Vahrameev, Babenko20}.
    
    Set  
    $d(x,y)=d_{\RR,\RR}(h_X(x(\cdot), y(\cdot)), 0)$. We obtain a  $C(T,\mathbb{R})$-valued distance in $C(T ,X)$. With such distance function, the pair $(C(T ,X), d)$ becomes and $FM$-distance space.
    
{\bf 4. Cartesian product of $FM$-distance spaces.} Let $(X,d_{X,M})$ be an $FM$-distance space. On $X^m$, $m\in\NN$, one can define different  $M$-distances to make $X^m$ an $FM$-distance space. For example one can set
    $$
    d_{X^m,M}^\Sigma((x_1,\ldots,x_m),(y_1,\ldots,y_m))=\sum_{k=1}^md_{X,M}(x_k,y_k).
    $$
    Instead of the distance $d^\Sigma_{X^m,M}$ one can consider (if $M$ satisfies the upper Riesz property and from  $\{x^n\}, \{y^n\}\in N(M)$ it follows $\{x^n\vee y^n\}\in N(M)$) the distance 
    $$
    d_{X^m,M}^\vee((x_1,\ldots,x_m),(y_1,\ldots,y_m))=\bigvee_{k=1}^md_{X,M}(x_k,y_k).
    $$
    Another distance function can be defined as
    $$
    d^m_{X^m,M^m}((x_1,\ldots,x_m),(y_1,\ldots,y_m))=(d_{X,M}(x_1,y_1),\ldots,d_{X,M}(x_m,y_m)).
    $$
    Here $M^m$ is a monoid with coordinate-wise addition and order.

\section{Contraction mapping theorems}\label{s::contractionTheorem}
\subsection{A contraction mapping theorem in FM-distance spaces}
\begin{definition}
 An $FM$-distance space $(X,d_{X,M})$ will be called complete, if each Cauchy sequence converges to an element from the space $X$.
\end{definition}

\begin{definition}
Denote by $\mathcal{L}(N(M))$ the class of operators $\lambda\colon M_+\to M_+$ that satisfy the following properties:
\begin{enumerate}
    \item $\lambda $ is non-decreasing, i.~e. if $x\leq y$, then $\lambda(x)\leq \lambda(y)$;
    \item If $\{ x_n\}\in N(M)$, then $\{\lambda( x_n)\}\in N(M)$;
    \item For arbitrary $\alpha\in M_+$, \begin{equation}\label{strongConvergence}
  \sum_{n=1}^\infty \lambda^n(\alpha) 
\end{equation}
is a Cauchy series.
\end{enumerate}
\end{definition}

\begin{theorem}\label{th::generalContraction}
Let a complete $FM$-distance space $(X,d_{X,M})$
and  $f\colon X\to X$ be given. Assume that  there exists $\lambda\in \mathcal{L}(N(M))$ such that 
\begin{equation}\label{strongContraction}
    d_{X,M}(f(x),f(y))\leq \lambda(d_{X,M}(x,y)) \text { for all } x,y\in X.
\end{equation}
Then  $f$ has a unique fixed point, which is the limit of the sequence $x_n = f(x_{n-1})$, $n\in\NN$, with arbitrarily chosen $x_0\in X$.
\end{theorem}

\begin{proof}
Let $x_0\in X$. Then for $n\in\NN$, due to monotonicity of $\lambda$, one has 
\begin{equation}\label{distDelta}
d_{X,M}(x_n,x_{n+1})
\leq 
\lambda(d_{X,M}(x_{n-1},x_{n}))
 \leq \ldots\leq
\lambda^n(d_{X,M}(x_{0},x_{1})). 
\end{equation}
Since $\sum_{n=1}^\infty \lambda^n(d_{X,M}(x_{0},x_{1}))$ is a Cauchy series, due to~\eqref{distDelta}, we obtain that the series $$
\sum_{k=1}^{\infty} d_{X,M}(x_n, x_{n+1})
$$
is also a Cauchy series. Lemma~\ref{lem::connection} implies that  $\{ x_n\}$ is a Cauchy sequence, and hence has a limit
$$
x_n=f(x_{n-1})\to \overline{x},\text{ as } n\to \infty.    
$$
Next we show that $f(\overline{x})=\overline{x}$. 
Since $\{d_{X,M}(x_n,\overline{x})\}\in N(M)$, 
$$
d_{X,M}(x_n,f(\overline{x})) 
= 
d_{X,M}(f(x_{n-1}),f(\overline{x}))
\leq 
\lambda (d_{X,M}(x_{n-1},\overline{x})),
$$
and by assumptions on $\lambda$, $\{\lambda (d_{X,M}(x_{n-1},\overline{x}))\}\in N(M)$, then due to the Fr\'{e}chet-Wilson property, we obtain that 
$$
\{d_{X,M}(f(\overline{x}), \overline{x})\}\in N(M)
$$
and hence $f(\overline{x}) = \overline{x}$.

Assume that
there are two points $\overline{x},\underline{x}\in X$  such that $f(\overline{x}) =\overline{x}$ and $f(\underline{x}) = \underline{x}$.
 Then for all $n\in \NN$,
 {
\begin{equation}\label{lambda^n}
d_{X,M}(\overline{x},\underline{x})
=
d_{X,M}(f^n(\overline{x}),f^n(\underline{x}))
 \leq 
\lambda^n(d_{X,M}(\overline{x},\underline{x})).
\end{equation}
}
From the definition of a Cauchy series, one obtains
$$
\{\lambda^n(d_{X,M}(\overline{x},\underline{x}))\}\in N(M). 
$$
Hence, due to arbitrariness of $n\in\NN$ in~\eqref{lambda^n},
 we obtain $\{d_{X,M}(\overline{x},\underline{x})\}\in N(M)$, thus  $d_{X,M}(\overline{x},\underline{x}) = \theta_M$ and $\overline{x}=\underline{x}$.
The theorem is proved.
\end{proof}

Using the same arguments one can relax the requirement on $\lambda$, if instead impose {the following }additional restriction on $f$:
{
\begin{equation}\label{orbitalContinuity}
    \text{If } x\in X, f^n(x)\to a \text { as } n\to\infty, \text{ then } f(a) =a.
\end{equation}
Such continuity requirement, in particular, follows from orbital continuity (see e.g.~\cite[Definition 2.3]{NietoPousoRodriges}) and uniqueness of the limit.
}


\begin{theorem}\label{th::continuousF}
Let a complete $FM$-distance space $(X,d_{X,M})$
and $f\colon X\to X$ {that satisfies~\eqref{orbitalContinuity}} be given. Assume that a mapping $\lambda\colon M_+\to M_+$ is non-decreasing and $x_0\in X$ is such that for the sequence $x_n = f(x_{n-1})$, $n\in\NN$, one has
\begin{equation}\label{contraction}
   d_{X,M}(x_n,x_{n+1})\leq \lambda(d_{X,M}(x_{n-1},x_n)) \text{ for all } n\in\NN,
\end{equation}
and
\begin{equation}\label{convergence}
\sum_{n=1}^\infty \lambda^n(d_{X,M}(x_0, x_1))
\end{equation}
is a Cauchy series. Then the mapping $f$ has a fixed point, which is the limit of the sequence $\{x_n\}$. If instead of conditions~\eqref{contraction} and~\eqref{convergence}, stronger conditions~\eqref{strongContraction} and~\eqref{strongConvergence} hold, then  $f$ has a unique fixed point.
\end{theorem}

\subsection{A contraction mapping theorem in partially ordered FM-distance spaces}

\begin{definition}
An $FM$-distance space $(X,d_{X,M})$ is called partially ordered, if $X$ is a partially ordered set.
\end{definition}

\begin{definition}
The $N(M)$-convergence in a partially ordered  $FM$-distance space is called upper regular, if for each non-decreasing sequence $$x_1\leq x_2\leq\ldots\leq x_n\leq\ldots$$ such that $\lim_{n\to\infty} x_n = x$, one has $x_n\leq x$ for all $n\in\NN$. In this case we also say that the convergence in the monoid $M$ is upper regular. 

Similarly one can introduce the notion of lower regularity. 
 If the convergence is both upper and lower regular, then we say that it is regular.
\end{definition}

\begin{theorem}\label{th::monotoneMappingContraction}
Let $(X,d_{X,M})$ be a partially ordered complete $FM$-distance space with upper regular family of $N(M)$-convergent sequences. Let $\lambda\in\mathcal{L}(N(M))$ and an operator $f\colon X\to X$ be  non-decreasing and such that for all $x,y\in X$ such that  $x\leq y$ condition~\eqref{strongContraction} holds. Assume that there exists $x_0\in X$ such that $x_0\leq f(x_0)$.
Then the operator $f$ has at least one fixed point, which is the limit of the sequence  $\{x_n = f(x_{n-1})\}$, $n\in\NN$. 

It is unique inside the set of points $x\in X$ that are comparable with $x_0$. In the case, when $X$ satisfies the upper Reisz property, the fixed point is unique in the whole space $X$.
\end{theorem}
\begin{proof}
From the conditions of the theorem, we obtain that $x_0 \leq f(x_0) = x_1$, $x_1 = f(x_0)\leq f(x_1) = x_2$, and so on. Hence $x_n\leq x_m$ for arbitrary $m\ge n\in \ZZ_+$. 

Therefore inequalities~\eqref{distDelta} hold,
the series $$
\sum_{k=1}^{\infty} d_{X,M}(x_n, x_{n+1})
$$ is a Cauchy series, and due to Lemma~\ref{lem::connection},  $\{ x_n\}$ is a Cauchy sequence. 
Completeness of $(X,d_{X,M})$ implies that there exists $x\in X$ such that $x_n\to x$ as $n\to\infty$. Next we show that  $f(x)=x$.

Since the convergence in the partially ordered set $X$ is upper regular by the assumptions of the theorem, we obtain $x\geq x_n$ for all $n\in \mathbb{Z}_+$. Hence, due to monotonicity of $f$, $x_n\leq f(x)$ for all $n\in \mathbb{Z}_+$. Thus, 
$$
d_{X,M}(x_n,f(x))=d_{X,M}(f(x_{n-1}),f(x))\le \lambda(d_{X,M}(x_{n-1},x)).
$$
Since $x_n\to x$ as $n\to\infty$ and $\lambda\in \mathcal{L}(N(M))$, we obtain
$$
\{ d_{X,M}(x_n, x)\}\in N(M)\text{ and } \{ d_{X,M}(x_n, f(x))\}\in N(M).
$$
The Fr\'{e}chet-Wilson condition  implies that $\{ d_{X,M}(x, f(x))\}\in N(M)$. Hence $d_{X,M}(x, f(x))=\theta_M$ and $x=f(x)$.

If $\overline{x}$ is a fixed point of $f$ that is comparable to $x_0$, then, due to monotonicity of $f$, it is comparable to each $x_n$, $n\in\ZZ_+$. Hence 
$$
d_{X,M}(x_{n}, \overline{x}) = 
d_{X,M}(f(x_{n-1}), f^n(\overline{x})) 
\leq
\lambda^n(d_{X,M}(x_0, \overline{x}) ).
$$
Thus $\{d_{X,M}(x_{n}, \overline{x})\}\in N(M)$. Analogously, if $\underline{x}$ is also a fixed point of $f$ that is comparable to $x_0$, then $\{d_{X,M}(x_{n}, \underline{x})\}\in N(M)$. Due to the Fr\'{e}chet-Wilson property, we obtain $\{d_{X,M}(\overline{x}, \underline{x})\}\in N(M)$, which implies $\overline{x} = \underline{x}$.

Finally, assume that there are two fixed points $\overline{x},\underline{x}$ of the mapping $f$ and $X$ satisfies the upper Riesz property. 

Since $d_{X,M}(f^n(\overline{x}), f^n(x_0\vee \overline{x}))\leq \lambda^n(d_{X,M}(\overline{x}, x_0\vee \overline{x}))$ and  
$$
\sum_{k=1}^\infty\lambda^n(d_{X,M}(\overline{x}, x_0\vee \overline{x}))
$$ 
is a Cauchy series, we obtain that $\{d_{X,M}(f^n(\overline{x}), f^n(x_0\vee \overline{x}))\}\in N(M)$. Analogously $\{d_{X,M}(f^n(x_0\vee \overline{x}), f^n(x_0) )\}\in N(M)$.
Applying the Fr\'{e}chet-Wilson property 
we obtain
$$
\{d_{X,M}(\overline{x},f^n(x_0))\} 
= 
\{d_{X,M}(f^n(\overline{x}),f^n(x_0))\}\in N(M).
$$ 
Analogously 
$$
\{d_{X,M}(f^n(x_0), \underline{x})\} 
= 
\{d_{X,M}(f^n(x_0),f^n(\underline{x}))\}\in N(M).
$$ 
Finally, applying the  Fr\'{e}chet-Wilson property once again,
we obtain 
$$\{d_{X,M}(\underline{x},\overline{x})\}\in N(M),$$ and hence $\overline{x} = \underline{x}$.
The theorem is proved.
\end{proof}

Analogously to Theorem~\ref{th::continuousF}, we can relax conditions on $\lambda$ by requiring continuity of $f$.
\begin{theorem}
Let a complete $FM$-distance space $(X,d_{X,M})$
and a non-dec\-reasing mapping $f\colon X\to X$ {that satisfies~\eqref{orbitalContinuity}} be given. Assume that there exists $x_0\in X$ such that $x_0\leq f(x_0)$ and for the sequence $\{x_n = f(x_{n-1})\}$ conditions~\eqref{contraction} and~\eqref{convergence} hold. Then the mapping $f$ has a fixed point, which is the limit of the sequence $\{x_n\}$.
\end{theorem}

 \subsection{Multiple fixed point theorems.}
Let $(X,d_{X,M})$ be a partially ordered $FM$-space and $P\subset \{ 1,\ldots,m\}$. Define a partial order in the set $X^m$, setting for $x=(x_1,\ldots,x_m)$ and $ y=(y_1,\ldots,y_m)$
$$
x\preccurlyeq_{X,P} y \text{ if and only if }  x_i\le y_i \text{ for } i\in P, \text{ and } y_i\le x_i\text{ for } i\notin P.
$$ 

Let a set of mappings
$$\sigma_i\colon\{1,\ldots,m\}\to \{1,\ldots,m\},\quad i=1,\ldots ,m$$
 be given, $\Bar{\sigma}=(\sigma_1,\ldots,\sigma_m)$ and $f\colon X^m\to X$. The mappings $\Bar{\sigma}$ and $f$ generate the mapping
$\Bar{\sigma} f\colon X^m\to X^m$
in the following way:
$$(\Bar{{\sigma}} f)(x_1,\ldots,x_m)=(y_1,\ldots,y_m),$$
where
$$y_i=f(x_{\sigma_i (1)},\ldots, x_{\sigma_i (m)}).$$

\begin{definition}
An element $x\in X^m$ is called a $\Bar{\sigma}$-multiple fixed point of the mapping $f\colon X^m\to X$, if $x=\Bar{\sigma}f(x)$.
\end{definition}

\begin{definition}
A mapping $f\colon X^m\to X^m$ is called $P$-monotone, if
$$
x\preccurlyeq_{X,P} y\text{ implies } f(x)\preccurlyeq_{X,P} f(y).
$$
\end{definition}

Let $M$ be a partially ordered monoid and $N(M)$ be a family of null sequences. In  $M^m$ define the coordinate-wise order and the family  $N(M^m)$ of null sequences, setting
$$
N(M^m)=\{ (x_1^n,\ldots,x_m^n)\colon \{ x_i^n\}\in N(M),\; i=1,\ldots,m\}.
$$
In $X^m$ define a distance function, setting
$$
d^m_{X^m,M^m}(x,y)=(d_{X,M}(x_1,y_1),\ldots,d_{X,M}(x_m,y_m)).
$$
The set $X$ with the above distance and partial order $\preccurlyeq_{X,P}$ will be denoted by $(X,\preccurlyeq_{X,P}, d^m_{X^m,M^m})$.

It is clear, that if the space $(X,d_{X,M})$ is $N(M)$-complete, then the space $(X^m,\preccurlyeq_{X,P},d^m_{X^m,M^m})$ is  $N(M^m)$-complete. It is also easy to see, that if the $N(M)$-convergence in $(X,d_{X,M})$ is regular, then  the $N(M^m)$-convergence in $(X^m,\preccurlyeq_{X,P},d^m_{X^m,M^m})$ is also regular.

Taking into account the above definitions and facts, Theorem~\ref{th::monotoneMappingContraction} implies the following theorem.
 \begin{theorem}\label{th::multipleFixedPoint}
Assume that an $FM$-space  $(X,d_{X,M})$ is $N(M)$-complete and the $N(M)$-convergence in it is regular. 
Let also a set $P\subset \{ 1,\ldots,m\}$ and mappings 
$$
\sigma_i\colon\{1,\ldots,m\}\to \{1,\ldots,m\},\quad i=1,\ldots ,m,$$
be given. Assume that a mapping $\overline{\sigma}f\colon X^m\to X$  is $P$-monotone and there exists an operator $\lambda\in\mathcal{L}(N(M^m))$ such that if  $x\preccurlyeq_{X,P} y$, then
$$
d^m_{X^m,M^m}(\overline{\sigma}f(x),\overline{\sigma}f(y))\le 
\lambda (d^m_{X^m,M^m}(x,y)).
$$
If there exists $x^0\in X^m$ such that  $x^0\preccurlyeq_{X,P}f(x^0)$, then there exists $x\in X^m$ such that  $x=\overline{\sigma}f(x)$, i.~e. the mapping $f$ has a  $\overline{\sigma}$-multiple fixed point, which is the limit of sequence $\{x^n = f(x^{n-1})\}$. 

It is unique inside the set of points $x\in X^m$ that are $P$-comparable with $x^0$. In the case, when $X^m$ satisfies the upper Reisz property, the fixed point is unique in the whole space $X^m$.
\end{theorem}

\begin{remark}
Instead of the distance $d^m_{X^m,M^m}$ one can consider the distances $d^\Sigma_{X^m,M}$ or (if $M$ satisfies the upper Riesz property and from  $\{x^n\}, \{y^n\}\in N(M)$ it follows $\{x^n\vee y^n\}\in N(M)$) $d_{X^m,M}^\vee$ (see Example~4 in Section~\ref{s::examples}). 
    
 It is easy to see that analogues of Theorem~\ref{th::multipleFixedPoint} hold for the spaces 
 $$(X^m,\preccurlyeq_{X,P},d^\Sigma_{X^m,M})\text{ and }(X^m,\preccurlyeq_{X,P},d^\vee_{X^m,M}).$$
\end{remark}

\section{Applications}\label{s::applications}

 \subsection{ Uniform space as a partial case of an M-metric space.}  We use the terminology from~\cite[Chapter~6]{Kelley}, see also~\cite[Chapter~8]{Engelking}.
  Let $X$ be a set and $\Phi$ be a separating uniform structure on it.
  Denote by $M$ the set of all 
  subsets of $X$ that contain the set $\Delta(X):=\{(x,x)\colon x\in X\}$. It becomes a partially ordered monoid, if we set  $\theta_M=\Delta(X)$,
  {
  \begin{multline*}
       A + B = A\circ B:=\{(x,y)\in X\times X\colon \text{there exists } z\in X \text{ such that } 
       \\
       (x,z)\in A \text{ and } (z,y)\in B\},   
  \end{multline*}
  }
   and define the order by inclusion. Obviously, $M_+ = M$.
  
  Denote by $\EE$ a symmetric base of entourages of the uniformity $\Phi$ and set 
  $$
    d_{X,M}(x,y) = \bigcap\left\{\varepsilon\in \EE\colon (x,y)\in \varepsilon\right\}.
  $$
  Next we show that $(X, d_{X,M})$ is an $M$-metric space (see Example~1 in Section~\ref{s::examples} for the definition).
  
  If $x = y$, then $(x,y) \in \varepsilon$ for all $\varepsilon\in \EE$ and hence $d_{X,M}(x,y) = \Delta(X)$ due to the separating axiom. Conversely, if $d_{X,M}(x,y) = \Delta(X)$, then $(x,y)\in \Delta(X)$ and hence $x = y$.
  
  From the definition of the function $d_{X,M}$, due to symmetricity of entourages from $\EE$, it follows that $d_{X,M}(x,y) = d_{X,M}(y,x)$.
  
  Finally, let $x,y,z\in X$. Then $(x,z)\in d_{X,M}(x,y)\circ d_{X,M}(y,z)$, and hence $d_{X,M}(x,z)\leq d_{X,M}(x,y) + d_{X,M}(y,z).$
  
  The family $\EE$ satisfies all properties of Definition~\ref{def::solidM+}, and hence the convergence in the space $(X,d_{X,M})$ coincides with the convergence in the topology induced by the uniform structure.

\subsection{Equations of the form $x=\Lambda(f,x)$}
Let  $X$ be an $FM$-distance space, $Y$ be a set. Consider an operator $\Lambda\colon Y\times X\to X$ and assuming that $z\in Y$ is known, we are interested in the question of existence and uniqueness of a  solution of the equation 
\begin{equation}\label{operatorEquation}
    x=\Lambda(z,x).
\end{equation}

Assume that there exists a non-decreasing mapping $\lambda\colon M_+\to M_+$ such that for all  $z\in Y$ and $x,y\in X$
$$
d_{X,M}(\Lambda (z,x),\Lambda (z,y))\leq \lambda(d_{X,M}(x,y))
$$
and the series 
\begin{equation}\label{**}
\sum_{n=1}^\infty \lambda^n(d_{X,M}(x,y))
\end{equation}
is a Cauchy series.
Applying Theorem~\ref{th::generalContraction} to the function $x\to \Lambda (z,x)$, we obtain that it has a unique fixed point, which is the limit of the sequence 
$$
x_0, x_1=\Lambda (z,x_0),\ldots, x_n=\Lambda (z,x_{n-1}), \ldots
$$
where $x_0$ is arbitrary point from  $X$. Hence equation~\eqref{operatorEquation} has a unique solution.

As an illustration, consider the case, when $M=\RR$,  instead of $X$ we consider the Cartesian product $(X^2,d_{X^2,\RR^2})$ of $FM$-spaces (see the example in Subsection~\ref{s::examples}) and assume that  $\lambda = \lambda_z$ is given by a matrix
$$
\begin{pmatrix}
a_z & b_z\\
c_z& d_z
\end{pmatrix}.
$$
Then series~\eqref{**} is a Cauchy series, provided  $\rho(\lambda_z)<1$, where $\rho(\lambda_z)$  is the spectral radius of the operator $\lambda_z$. The condition  $\rho(\lambda_z)<1$ is equivalent (see~\cite[\S 7.1]{Zabrejko}) to the following numeric inequalities:
$$
a_z<1, d_z<1,\text{ and } b_zc_z<(1-a_z)(1-d_z),
$$
which allows $b_z$ to be large.

\subsection{Fredholm integral equations.}\label{s::fredholmEquation}
Let $T$ be a metric compact and $\mu$ be a measure defined on the $\sigma$-algebra of Borel sets of $T$. Let also {an} $L$-space {(the definition can be found in~\cite{Vahrameev, Babenko19, Babenko20,Babenko21})}  $(X,h_X)$ be given, and $C(T,X)$ be the space of continuous functions $x\colon T\to X$. In the space $C(T,X)$ consider a $C(T,\mathbb{R})$-valued metric, setting
$$
h_{C(T,X)}(x,y)=h_X(x(\cdot),y(\cdot)),
$$
where $C(T,\RR)$ is considered as a partially ordered monoid with pointwise addition and partial order.
It is easy to see that the obtained space $C(T,X)$ is complete whenever $X$ is complete.

We are interested, whether a solution of the equation
$$
x(t)=f(t)+\int_Tg(t,s,x(s))d\mu(s)
$$
exists and is unique, where
$g\in C(T\times T\times X, X)$ and $f\in C(T,X)$ are given and $x$ is to be found.

Assume that for arbitrary $x,y\in C(T,X)$ and $t,s\in T$, 
$$
h_{X}(g(t,s,x(s)),g(t,s,y(s)))\le Q(t,s)h_X(x(s),y(s)),
$$
where $Q\in C(T\times T,\mathbb{R}_+)$. If, for example, $g(t,s,x(s))=K(t,s)x(s)$, where $K\in C(T\times T,\mathbb{R}_+)$, i.~e. the considered equation is linear with non-negative kernel, then $Q(t,s)=K(t,s)$. 

For the operator $A\colon C(T,X)\to C(T,X)$
$$
Ax(\cdot)=f(\cdot)+\int_Tg(\cdot,s,x(s))d\mu(s)
$$
we have
$$
h_{C(T,X)}(Ax,Ay)\le \lambda(h_{C(T,X)}(x,y)),
$$
where $\lambda\colon C(T,\RR_+)\to C(T,\RR_+)$ is the linear integral operator with the kernel $Q(t,s)$:
$$
\lambda(x(\cdot))(t)=\int_TQ(t,s)x(s)d\mu(s).
$$
For arbitrary $n\in\mathbb{N}$
$$
\lambda^n(x(\cdot))(t)=\int_TQ_n(t,s)x(s)d\mu(s),
$$
where $Q_1(t,s)=Q(t,s)$ and $Q_n(t,s)=\int_TQ_{n-1}(t,u)Q_1(u,s)d\mu(u)$.

Theorem~\ref{th::generalContraction} is applicable to the operator $A$ if for each function $x\in C(T,X_+)$ the series
$$
\sum_{n=1}^\infty \int_TQ_n(t,s)x(s)d\mu(s)
$$
converges in the space $(C(T,X),h_{C(T,X)})$.
It is easy to see, that it is sufficient to require that the series
$$
\sum_{n=1}^\infty \int_TQ_n(t,s)d\mu(s)
$$
converges in the space $(C(T,X),h_{C(T,X)})$. Thus, if the latter series converges, then the considered Fredholm equation has a unique solution in the space $C(T,X)$. It is easy to see that this condition is significantly more general than the one from~\cite{Babenko19}.

\appendix
\section{Proofs of Lemmas~\ref{l::eCauchySeries} and~\ref{l::eFMCondition}. }\label{a::eNullSeqFamily}

\subsection{Proof of Lemma~\ref{l::eCauchySeries}}
\begin{proof}
Assume that $\sum_{n=1}^\infty x_n$ is a Cauchy series but there exists $\varepsilon\in\EE$ such that for each $N\in\NN$ there exist $m\geq n\geq N$ such that   inequality~\eqref{eFundamelity} does not hold. Then one can build increasing sequences of natural numbers $\{m_k\}$ and $\{n_k\}$, $n_k\leq m_k$ for all $k\in \NN$ such that the inequality  $ \sum_{s=n_k}^{m_k} x_s < \varepsilon$ does not hold for all $k\in\NN$. However, this contradicts to the assumption 
\begin{equation}\label{generalFundamelity}
 \left\{\sum_{s=n_k}^{m_k} x_s\right\}\in N_{\EE}(M).   
\end{equation}

If inequality~\eqref{eFundamelity} holds, then for arbitrary $\varepsilon\in \EE$ and arbitrary increasing sequences of natural numbers $\{m_k\}$, $\{n_k\}$ one has $\sum_{s=n_k}^{m_k} x_s<\varepsilon$ whenever $k$ is such that  $n_k>N$, i.~e.~\eqref{generalFundamelity} holds.
 \end{proof}
 
 \subsection{Proof of Lemma~\ref{l::eFMCondition}}
  \begin{proof}
 Let  a sequence $\{x_n\}\subset X$, increasing sequences of natural numbers $\{n_k\}$ and $\{m_k\}$ such that $m_k> n_k$ for all $k\in\NN$, and  
$$
\left\{\sum_{s = n_k}^{m_k-1} d_{X,M}(x_s,x_{s+1})\right\}\in N_{\EE}(M)
$$
be given. For each $\varepsilon\in \EE$ choose $\delta\in\EE$ according to the formulated condition, and for the found $\delta$ there exists $N\in\NN$ such that for all $k\geq N$ one has $\sum_{s = n_k}^{m_k-1} d_{X,M}(x_s,x_{s+1})<\delta$. Then for all $k\geq N$, $d_{X,M}(x_{n_k},x_{m_k})<\varepsilon$ and hence
$$
\left\{ d_{X,M}(x_{n_k},x_{m_k})\right\}\in N(M),
$$
thus $(X,d_{X,M})$ is an $FM$-space and sufficiency is proved.

Assume that for some $\varepsilon\in\EE$ and arbitrary $\delta\in \EE$ there exists $n_\delta\in\NN$ and $x_1^\delta,\ldots, x_{n_\delta}^\delta\in X$ such that 
$$
\sum\limits_{k=1}^{n_\delta-1} d_{X,M}(x_k^\delta,x_{k+1}^\delta)< \delta,
$$
but the inequality 
\begin{equation}\label{notInN(M)}
    d_{X,M}(x_1^\delta, x_{n_\delta}^\delta)<\varepsilon
\end{equation}
{does not hold.}
 Consider a sequence $\{\delta_k\}\in N_{\EE}(M)$, $\delta_k\subset \EE$, and the sequence 
$$\{y_s\} 
= \left\{
x_1^{\delta_1},\ldots, x_{n_{\delta_1}}^{\delta_1},
x_1^{\delta_2},\ldots, x_{n_{\delta_2}}^{\delta_2},
\ldots
\right\}.
$$
Applying the strong Fr\'{e}chet-Wilson property to the sequences $u_1 = 1$, $u_k= u_{k-1} + n_{\delta_{k-1}}$, and $v_1 = n_{\delta_1}$, $v_k = v_{k-1} + n_{\delta_{k-1}}$, $k\geq 2$, we obtain that 
$$
\left\{d_{X,M}(y_{u_k}, y_{v_k}) \right\}
=\left\{d_{X,M}\left(x_1^{\delta_k}, x_{n_{\delta_k}}^{\delta_k}\right) \right\}
\in N_{\EE}(M),
$$
which contradicts to {the fact that inequality~\eqref{notInN(M)} does not hold}.
 \end{proof}

\bibliographystyle{plain}
\bibliography{bibliography}

\begin{thebibliography}{10}

\bibitem{Agarwal}
R.~P. Agarwal, M.~Meehan, D.~O’Regan,
\newblock {\em Fixed Point Theory and Applications},
\newblock Cambridge Tracts in Math.,
\newblock Cambridge Univ. Press, 2001,
\newblock \doi{10.1017/CBO9780511543005}. $\,$

\bibitem{Babenko19}
V.~Babenko,
\newblock Calculus and nonlinear integral equations for functions with values
  in {L}-spaces,
\newblock {\em Anal Math}, {\bf 45}:727--755, 2019. $\,$

\bibitem{Babenko20}
V.~Babenko, V.~Babenko, O.~Kovalenko,
\newblock Optimal recovery of monotone operators in partially ordered
  {L}-spaces,
\newblock {\em Numer. Funct. Anal. Optim.}, pp. 1--25, 2020,
\newblock \doi{10.1080/01630563.2020.1775251}. $\,$

\bibitem{Babenko21}
V.~Babenko, V.~Babenko, O.~Kovalenko, M.~Polishchuk,
\newblock Optimal recovery of operators in function {L}-spaces,
\newblock {\em Anal Math}, {\bf 47}:13–32, 2021,
\newblock \doi{10.1007/s10476-021-0065-y}. $\,$

\bibitem{carl2014fixed}
S.~Carl, S.~Heikkil{\"a},
\newblock {\em Fixed Point Theory in Ordered Sets and Applications: From
  Differential and Integral Equations to Game Theory},
\newblock Springer New York, 2014. $\,$

\bibitem{Choban}
M.~Choban,
\newblock Fixed points of mappings defined on spaces with distance,
\newblock {\em Carpathian J. Math}, {\bf 32}(2):173 -- 188, 2016. $\,$

\bibitem{ChobanBerinde}
M.~Choban, V.~Berinde,
\newblock A general concept of multiple fixed point for mappings defined on
  spaces with distance,
\newblock {\em Carpathian J. Math}, {\bf 33}(3):275 -- 286, 2017. $\,$

\bibitem{Engelking}
R.~Engelking,
\newblock {\em General topology},
\newblock Panstwowe Wydavnictwo Naukowe, Warszawa, 1977. $\,$

\bibitem{Granas}
A.~Granas, J.~Dugundji,
\newblock {\em Fixed Point Theory},
\newblock Cambridge Tracts in Math.,
\newblock Cambridge Univ. Press, 2004. $\,$

\bibitem{guo2004partial}
D.~Guo, Y.J. Cho, J.~Zhu,
\newblock {\em Partial Ordering Methods in Nonlinear Problems},
\newblock Nova Science Publishers, 2004. $\,$

\bibitem{Jancovic}
S.~Jankovi\'{c}, Z.~Kadelburg, S.~Radenovi\'{c},
\newblock On cone metric spaces: A survey,
\newblock {\em Nonlinear Anal Theory Methods Appl}, {\bf 74}(7):2591 -- 2601,
  2011,
\newblock \doi{10.1016/j.na.2010.12.014}. $\,$

\bibitem{Kantorovich}
L.~V. Kantorovich,
\newblock To the general theory of operations in semiordered spaces,
\newblock {\em Dokl. Akad. Nauk SSSR}, {\bf 1}(7):271--274, 1936. $\,$

\bibitem{KVP}
L.~V. Kantorovich, B.~Z. Vulikh, A.~G. Pinsker,
\newblock {\em Functional analysis in semiordered spaces},
\newblock Gostekhizdat, Moscow, 1950. $\,$

\bibitem{Kelley}
J.~L. Kelley,
\newblock {\em General Topology},
\newblock Graduate Texts in Mathematics,
\newblock Springer New York, 1975. $\,$

\bibitem{Kirk}
W.~Kirk, N.~Shahzad,
\newblock {\em Fixed Point Theory in Distance Spaces},
\newblock Springer New York, 2014. $\,$

\bibitem{Kurepa}
G.~Kurepa,
\newblock Tableaux ramifies d'ensembles. espaces pseudodistancies,
\newblock {\em C. R. Acad. Sci. Paris, Ser. A}, {\bf 198}:1563--1565, 1934.
  $\,$

\bibitem{Kusraev}
A.~G Kusraev,
\newblock Kantorovich spaces and the metrization problem,
\newblock {\em Sibirsk. Mat. Zh.}, {\bf 34}(4):108--116, 1993. $\,$

\bibitem{Lang}
S.~Lang,
\newblock {\em Algebra},
\newblock Graduate Texts in Mathematics,
\newblock Springer New York, 2005. $\,$

\bibitem{NietoPousoRodriges}
J.~Nieto, R.~L. Pouso, R.~Rodrigues-Lopes,
\newblock Fixed point theorems in ordered abstract spaces,
\newblock {\em Proc. of the American Math. Soc.}, {\bf 135}(8):2505--2517,
  August 2007. $\,$

\bibitem{Radenovic_Kadelburg}
S.~Radenovic, Z.~Kadelburg,
\newblock Quasi-contractions on symmetric and cone symmetric spaces,
\newblock {\em Banach J. Math. Anal.}, {\bf 5}(1):38--50, 2011. $\,$

\bibitem{Reem}
D.~Reem, S.~Reich,
\newblock Zone and double zone diagrams in abstract spaces,
\newblock {\em Colloquium Math.}, {\bf 115}:129--145, 2009. $\,$

\bibitem{tarski}
A.~Tarski,
\newblock A lattice-theoretical fixpoint theorem and its applications,
\newblock {\em Pac J Math}, {\bf 5}(2):285–309, 1955. $\,$

\bibitem{Taylor}
W.~W. Taylor,
\newblock Fixed-point theorems for nonexpansive mappings in linear topological
  spaces,
\newblock {\em J. Math. Anal. Appl.}, {\bf 40}(1):164 -- 173, 1972,
\newblock \doi{10.1016/0022-247X(72)90040-6}. $\,$

\bibitem{Vahrameev}
S.~A. Vahrameev,
\newblock {\em Applied Mathematics and Mathematical Software of Computers},
\newblock M.: MSU Publisher, 1980. (in Russian)

\bibitem{Wilson}
W.~A. Wilson,
\newblock On semi-metric spaces,
\newblock {\em Am. J. Math.}, {\bf 53}(2):361--373, 1931. $\,$

\bibitem{Zabrejko}
P.~P. Zabrejko,
\newblock K-metric and {K}-normed linear spaces: survey,
\newblock {\em Collect. Math.}, {\bf 48}(4-5-6):825--859, 1997. $\,$

\end{thebibliography}

\end{document}